\documentclass[12pt]{amsart}
\usepackage{amsmath, amssymb}

\usepackage{hyperref}
\hypersetup{%
 colorlinks=true,
 linkcolor=blue,
}

\numberwithin{equation}{section}

\newtheorem{theorem}{Theorem}[section]
\newtheorem{lemma}[theorem]{Lemma}

\newtheorem{conjecture}[theorem]{Conjecture}
\newtheorem{goal}[theorem]{Goal}

\theoremstyle{definition}

\newtheorem*{rem-nonum}{Remark}
\newtheorem*{remarks-nonum}{Remarks}

\newcommand{\bbC}{{\mathbb C}}

\newcommand{\bbZ}{{\mathbb Z}}

\def\Imag{\operatorname{Im}}
\def\Log{\operatorname{Log}}
\def\Real{\operatorname{Re}}

\begin{document}

\baselineskip=17pt

\title[Smyth's Conjecture on Mahler Measure\ldots]{Smyth's Conjecture on the Mahler Measure of non-reciprocal trinomials of height $1$}

\author{Paul M Voutier}
\address{London, UK}
\email{Paul.Voutier@gmail.com}

\begin{abstract}
We prove a conjecture of Smyth on the Mahler measure of non-reciprocal trinomials
of height $1$, determining exactly when their Mahler measures are greater than or
less than $M(x+y+1)=\rho=1.381356\ldots$, conjectured to be the smallest limit
point of $M(P)$ for non-reciprocal polynomials. Conjecturally, our result gives
the full spectrum of $M(P)<\rho$ for all non-reciprocal polynomials, not just
for trinomials.
\end{abstract}

\maketitle

\section{Introduction}

Consider
\[
P(z) = a_{n} \left( z-\alpha_{1} \right) \cdots \left( z-\alpha_{n} \right)
\in \bbZ[z],
\]
with $a_{n} \neq 0$. Writing $P(z)=\sum_{i=0}^{n} a_{i}z^{i}$, we define the
\emph{height of $P(z)$} to be $\max_{i=0,\ldots, n} \left| a_{i} \right|$.

The \emph{Mahler measure, $M(P)$, of $P(z)$} is defined by
\[
M(P) = \left| a_{n} \right| \prod_{i=1}^{n} \max \left\{ 1, \left| \alpha_{i} \right| \right\}.
\]

If $P(z)$ is the minimal polynomial of an algebraic number $\alpha$, we also say
that $M(\alpha)=M(P)$.

The Mahler measure is related to the absolute Weil height, $h(\alpha)$, by the
formula $\log M(\alpha)=nh(\alpha)$, where $n$ is the degree of $\alpha$. Both
are very important in the study of Diophantine problems, algebraic number theory
and many related fields.

A well-known conjecture, due to Lehmer \cite{Le}, asks if there exists an absolute
constant, $\mu>1$, such that if $M(P) \neq 1$, then $M(P) \geq \mu$. It is believed
that the answer is yes. Furthermore, the best possible $\mu$ is conjectured to
be $1.1762808\ldots=M(P)$ for $P(z)=z^{10}+z^{9}-z^{7}-z^{6}-z^{5}-z^{4}-z^{3}+z+1$,
an example known to Lehmer.
The best results to date originate from the ideas and work of Dobrowolski \cite{Do}
and show
that $\log M(P) \geq C \left( \log \log n / \log n \right)^{3}$ for some $C>0$
($C=1/4$ is the best known value for all $n \geq 2$).

Lehmer's polynomial is a reciprocal polynomial. That is, $P(z)=\pm z^{\deg(P)}P(1/z)$.
For non-reciprocal polynomials, we have better results. Smyth \cite{Smyth1}
proved that $M(P) \geq M \left( z^{3}-z-1 \right)=1.324717\ldots$ for irreducible
non-reciprocal polynomials $P$ with $P(0), P(1) \neq 0$.

\subsection{Conjectures}

Given such a result, it is natural to investigate the spectrum of values of
the Mahler measure for non-reciprocal polynomials.
In this direction, we have the following conjecture, due to Boyd \cite[bottom of page~458]{Boyd}.

\begin{conjecture}
\label{conj:1.1}
The smallest limit point of $M(P)$ for non-reciprocal polynomials $P(z)$ is
$M(x+y+1)=\rho=1.381356444518\ldots$. 
\end{conjecture}

Note that the Mahler measure can also be defined for polynomials in several
variables (see \cite{Mah}).

If Lehmer's conjecture is true (with $\mu=1.1762808\ldots$), then any polynomial,
$P(z)$, with $M(P)<\rho$ can have only one noncyclotomic irreducible factor.
Lehmer's conjecture is still unproven. Nonetheless, it suggests the following.

\begin{goal}
Find all irreducible non-reciprocal polynomials $P(z) \in \bbZ[z]$ with
$P(0), P(1) \neq 0$ and $M(P) < \rho$.
\end{goal}

Flammang \cite[p.~232]{F1} has proposed a conjecture about such polynomials.

\begin{conjecture}[V. Flammang]
\label{conj:1.2}
If $P(z) \in \bbZ[z]$ is an irreducible non-reciprocal polynomial with $M(P)<\rho$,
$P(0), P(1) \neq 0$, and $P \neq P_{1} \left( z^{m} \right)$ for any $m>1$ and
$P_{1}(z) \in \bbZ[z]$, then there exists a cyclotomic polynomial, $C_{P}(z)$,
such that $P(z) C_{P}(z)$ is a trinomial of the form $z^{n} \pm z^{k} \pm 1$.
\end{conjecture}

A first step towards the goal above would be to prove the following conjecture
due to C.J. Smyth (see Conjecture~1.1 of \cite{F1}).

\begin{conjecture}[C.J. Smyth]
\label{conj:1.3}
Let $k$ and $n$ be relatively prime positive integers with $0<k<n$.

\noindent
{\rm (a)} Put $P_{n,k}^{(1)}(z)=z^{n}+z^{k} + 1$. $M \left( P_{n,k}^{(1)} \right)<\rho$ if and only if $3|(n+k)$.\\
{\rm (b)} Put $P_{n,k}^{(2)}(z)=z^{n}-z^{k} + 1$. $M \left( P_{n,k}^{(2)} \right)<\rho$ with $n$ odd if and only if $3 \nmid (n+k)$.\\
{\rm (c)} Put $P_{n,k}^{(3)}(z)=z^{n}-z^{k} - 1$. $M \left( P_{n,k}^{(3)} \right)<\rho$ with $n$ even if and only if $3 \nmid (n+k)$.
\end{conjecture}

We show in Section~\ref{sect:notation} why these three cases cover all trinomials
of height $1$.

\subsection{Previous Work}

Part~(a) of Conjecture~\ref{conj:1.3} was proven for $n$ sufficiently large with
respect to $k$ by Duke \cite{Du}. In fact, he proved that
\begin{equation}
\label{eq:duke1}
\log M \left( P_{n,k}^{(1)}(z) \right)
= \log M(x+y+1)+\frac{\alpha(n+k)}{n^{2}}+O \left( \frac{k}{n^{3}} \right)
\end{equation}
where $\alpha(n+k)=-\pi \sqrt{3}/6$ if $3|(n+k)$ and $\alpha(n+k)=\pi \sqrt{3}/18$
otherwise. This generalised an earlier result due to Boyd \cite{Boyd} for $k=1$.

Flammang \cite{F1} considered parts~(b) and (c) of Conjecture~\ref{conj:1.3},
using the same technique as Duke. Like the results of Duke, her results (see her
Theorem~2.1) also imply that
parts~(b) and (c) of Conjecture~\ref{conj:1.3} are true once $n$ is sufficiently
large with respect to $k$. ``Sufficiently large'' can be made effective, and even
explicit, in both Duke's and Flammang's results.

\subsection{Results}

In this paper, we prove Smyth's conjecture (Conjecture~\ref{conj:1.3}) without
the conditions on the size of $n$ relative to $k$ that arise in \cite{Du,F1}.

\begin{theorem}
\label{thm:2.1}
Let $k$ and $n$ be relatively prime positive integers with $n \geq 2k$.
Write $n+k \equiv \epsilon \pmod{3}$, where $\epsilon \in \{ -1, 0, 1 \}$.

{\rm (a)} Conjecture~$\ref{conj:1.3}($a$)$ holds.
If $3|(n+k)$, then
\[
\log M \left( z^{n}+z^{k}+1 \right)
= \log M \left( x+y+1 \right)
-\frac{\pi\sqrt{3}}{6n^{2}}
-\frac{\pi\sqrt{3}\,k}{6n^{3}}
+O\!\left(\frac{k^{2}}{n^{4}}\right).
\]

If $3 \nmid (n+k)$, then
\[
\log M \left( z^{n}+z^{k}+1 \right)
= \log M \left( x+y+1 \right)
+ \frac{\pi\sqrt{3}}{18n^{2}}
+ \left( \frac{\pi\sqrt{3}\,k}{18} + \epsilon\frac{4\pi^{2}}{81} \right)\frac{1}{n^{3}}
+O\!\left(\frac{k^{2}}{n^{4}}\right),
\]

{\rm (b)} Suppose that $n$ is also odd. Conjecture~$\ref{conj:1.3}($b$)$ holds.
If $3|(n+k)$, then
\[
\log M \left( z^{n}-z^{k}+1 \right)
= \log M \left( x+y+1 \right)
+ \frac{\pi\sqrt{3}}{12n^{2}}
+ \frac{\pi\sqrt{3}\,k}{12n^{3}}
+O\!\left(\frac{k^{2}}{n^{4}}\right).
\]

If $3 \nmid (n+k)$, then
\[
\log M \left( z^{n}-z^{k}+1 \right)
= \log M \left( x+y+1 \right)
- \frac{\pi\sqrt{3}}{36n^{2}}
- \left( \frac{\pi\sqrt{3}\,k}{36} + \epsilon\frac{5\pi^{2}}{81} \right)\frac{1}{n^{3}}
+O\!\left(\frac{k^{2}}{n^{4}}\right).
\]

{\rm (c)} Suppose that $n$ is even. Conjecture~$\ref{conj:1.3}($c$)$ holds.
If $3|(n+k)$, then
\[
\log M \left( z^{n}-z^{k}-1 \right)
= \log M \left( x+y+1 \right)
+ \frac{\pi\sqrt{3}}{12n^{2}}
+ \frac{\pi\sqrt{3}\,k}{12n^{3}}
+O\!\left(\frac{k^{2}}{n^{4}}\right).
\]

If $3 \nmid (n+k)$, then
\[
\log M \left( z^{n}-z^{k}-1 \right)
= \log M \left( x+y+1 \right)
- \frac{\pi\sqrt{3}}{36n^{2}}
- \left( \frac{\pi\sqrt{3}\,k}{36} + \epsilon\frac{5\pi^{2}}{81} \right)\frac{1}{n^{3}}
+O\!\left(\frac{k^{2}}{n^{4}}\right).
\]

The constants involved in all the $O$ terms above are absolute and effective.
\end{theorem}

Although we start with the same integrals as in \cite{Du,F1}, we take a different
approach from these papers.
There are two key differences.

First, we move the two integrals in \eqref{eq:lem41-CG-3}
into the lower half-plane. Thus the oscillations that arise from the initial
integrals become exponential decay. This also allows us to simplify our argument,
avoiding repeated use of integration by parts to overcome the effects of these
oscillations.

Secondly, rather than evaluating the resulting integrals, we integrate using a
simpler function that allows us to approximate well the two integrands in our
moved integrals. This simpler integral is our $\Lambda_{n,k}^{(j)}$ defined at
the start of Section~\ref{sect:intermediate}. Another consequence of using
$\Lambda_{n,k}^{(j)}$ is that the $n^{-3}$ terms in the expansions in
Theorem~\ref{thm:2.1} can be naturally obtained with
only a little more work in Sections~\ref{sect:intermediate} and \ref{sect:5}.

\section{Preliminaries}
\label{sect:notation}

\subsection{Reductions}

The restriction to $n$ odd in part~(b) is not an actual restriction as for $n$
even, we must have $k$ odd and hence $(-z)^{n}-(-z)^{k}+1=z^{n}+z^{k}+1$. So it
is covered by part~(a).

Similarly, the restriction to $n$ even in part~(c) is not an actual restriction.
If $n$ and $k$ are both odd, then $-\left( (-z)^{n}-(-z)^{k}-1 \right)
=z^{n}-z^{k}+1$, so it is covered by part~(b).

If $n$ is odd and $k$ is even, then $-\left( (-z)^{n}-(-z)^{k}-1 \right)
=z^{n}+z^{k}+1$, so it is covered by part~(a).

The remaining trinomials, $z^{n}+z^{k}-1$, are also covered by the three cases
above. If $n$ is even, then
$k$ must be odd and $(-z)^{n}+(-z)^{k}-1=z^{n}-z^{k}-1$. This is covered by part~(c).

If $n$ and $k$ are both odd, then $-\left( (-z)^{n}+(-z)^{k}-1 \right)
=z^{n}+z^{k}+1$, so it is covered by part~(a).

If $n$ is odd and $k$ is even, then $-\left( (-z)^{n}+(-z)^{k}-1 \right)
=z^{n}-z^{k}+1$, so it is covered by part~(b).

If $d=\gcd(n,k)$, then $z^{n} \pm z^{k} \pm 1=Q \left( z^{d} \right)$, where $Q$
is a trinomial whose two non-zero exponents are relatively prime. Moreover,
$M \left( Q \left( z^{d} \right) \right)=M(Q)$. Hence, we may assume in this
work that $\gcd(n,k)=1$.

It suffices to consider $n \geq 2k$. We can do that without any loss of
information because by taking reciprocals, we can reduce the case of $n<2k$ to
$n>2k$ and the Mahler measure is unchanged. Also, the middle exponent is
$n-k<n/2$ and $n+(n-k) \equiv -(n+k) \pmod{3}$, so the divisibility by $3$
condition is preserved. For $n=2k$ with $\gcd(n,k)=1$, among
the polynomials considered in Conjecture~\ref{conj:1.3} only $z^{2}-z-1$ is not
a cyclotomic polynomial. It has $M \left( z^{2}-z-1 \right)=\left( 1+\sqrt{5} \right)/2
=1.61803\ldots$, so Conjecture~\ref{conj:1.3}(c) holds for it.

\subsection{Notation}

Let $\varepsilon_{k}$ be the coefficient of $z^{k}$ in $P_{n,k}^{(j)}(z)$
for $j \in \{ 1,2,3 \}$. Put
\[
\Delta_{n,k}^{(j)}= \log M \left( P_{n,k}^{(j)} \right) - \log M(1+x+y).
\]

For $s \geq 0$, define
\[
d(s)=\cosh s+i\sqrt{3}\sinh s,\quad
u(s)=e^{-(2n-k)s}\overline{d(s)}^{-k},\quad
v(s)=e^{-(2n-k)s}d(s)^{k}.
\]

Throughout the paper, $\Log$ denotes the principal branch of the logarithm function.
That is, writing $z=re^{i \theta}$ for $z \neq 0$ with $r=|z|$ and $-\pi<\theta \leq \pi$,
we have $\Log(z)=\log r +i\theta$.

Also, when $j=2$ we assume $n$ is odd, and when $j=3$ we assume $n$ is even. By
the reductions at the start of this section, this does not impose a restriction
on our results.

\section{The Exact Contour Identity}
\label{sect:lem1}

Our goal in this section is to prove Lemma~\ref{lem:actual-1}. We will prove
two intermediate lemmas first.

\begin{lemma}
\label{lem:new-1}
For relatively prime positive integers $n$ and $k$ with $n \geq 2k$, and
$j \in \{ 1,2,3 \}$,
\begin{align}
\label{eq:lem41-CG-3}
\Delta_{n,k}^{(j)}
= & \frac{2}{k} \left\{ \int_{0}^{1/3} \log \left| 1-\varepsilon_{k}e((n+k)/3)r_{1}(t) \right| \,dt \right. \\
& \hspace*{7.0mm} \left. + \int_{0}^{1/6} \log \left| 1-\varepsilon_{k}e(-(n+k)/3)r_{2}(t) \right| \,dt
\right\}, \nonumber
\end{align}
where
\begin{equation}
\label{eq:CG-4b}
r_{1}(t)=
\left( 2\cos \left( \frac{\pi}{3}-\pi t \right) \right)^{-k}
e^{-i\pi (2n-k)t}
\end{equation}
and
\begin{equation}
\label{eq:CG-5b}
r_{2}(t)=
\left( 2\cos \left( \frac{\pi}{3}+\pi t \right) \right)^{k}
e^{-i\pi (2n-k)t}.
\end{equation}
\end{lemma}

\begin{proof}
By Jensen's formula,
\[
\log M \left( P_{n,k}^{(j)} \right)
=\int_{0}^{1}\log \left| P_{n,k}^{(j)}(e(t)) \right| \,dt,
\]
where $e(u)=e^{2\pi i u}$.

Write $P_{n,k}^{(j)}(z)=z^{n}+\varepsilon_{k}z^{k}+\varepsilon_{0}$, where
$\varepsilon_{0}, \varepsilon_{k}=\pm 1$. Split the integral into the $k$
intervals $[\ell/k,(\ell+1)/k]$ for $\ell=0,\ldots,k-1$. Since $\gcd(n,k)=1$,
the numbers $e(n\ell/k)$, $0 \leq \ell<k$, are precisely the
$k$-th roots of unity. Hence, from the change of variables $t=(u+\ell)/k$,
\begin{align}
\label{eq:lem41-int1}
\log M \left( P_{n,k}^{(j)} \right)
&=\frac{1}{k} \int_{0}^{1} \log \left| \prod_{\ell=0}^{k-1} \left( \varepsilon_{0} + \varepsilon_{k} e(u)+e(nu/k)e(n\ell/k) \right) \right|\,du \nonumber \\
&=\frac{1}{k} \int_{0}^{1} \log \left| \left( \varepsilon_{0}+\varepsilon_{k}e(u) \right)^{k}-(-1)^{k} e(nu) \right|\,du.
\end{align}

We obtain the second equality from the identity $\prod_{\omega^{k}=1} (A+B\omega)
=A^{k}-(-B)^{k}$ applied with $A=\varepsilon_{0}+\varepsilon_{k}e(u)$ and
$B=e(nu/k)$.

We can write this as
\begin{equation}
\label{eq:lem41-int2}
\log M \left( P_{n,k}^{(j)} \right)
=\frac{1}{k} \int_{0}^{1} \log \left| \left( 1+e(u) \right)^{k}-\varepsilon_{k}(-1)^{k}e(nu) \right|\,du.
\end{equation}

This is immediate for $j=1$. I.e., $P_{n,k}^{(j)}(z)=P_{n,k}^{(1)}(z)=z^{n}+z^{k}+1$.

\vspace*{2.0mm}

For $j=2$, that is, $P_{n,k}^{(j)}(z)=P_{n,k}^{(2)}(z)=z^{n}-z^{k}+1$, apply the
translation $u \mapsto u+1/2$. Since $n$ is odd, we have $1-e \left( u+1/2 \right)=1+e(u)$
and $e \left( n \left( u+1/2 \right) \right)=-e(nu)$. So the integrand in \eqref{eq:lem41-int1}
becomes
\[
\left| (1+e(u))^{k}+(-1)^{k}e(nu) \right|
= \left| (1+e(u))^{k}-\varepsilon_{k}(-1)^{k}e(nu) \right|.
\]
Using the periodicity of the integrand, we can take the integral over $[0,1]$ to
obtain \eqref{eq:lem41-int2}.

\vspace*{2.0mm}

Lastly, for $j=3$, that is, $P_{n,k}^{(j)}(z)=P_{n,k}^{(3)}(z)=z^{n}-z^{k}-1$,
we have $n$ even and $k$ odd. Therefore,
\[
(-1-e(u))^{k}-(-1)^{k}e(nu)=-(1+e(u))^{k}+e(nu).
\]
Its absolute value is
\[
\left| (1+e(u))^{k}-e(nu) \right|
=\left| (1+e(u))^{k}-\varepsilon_{k} (-1)^{k}e(nu) \right|,
\]
since $\varepsilon_{k}=-1$ and $k$ is odd.

The integrand in \eqref{eq:lem41-int2} is invariant under $u \mapsto 1-u$ (the
quantity inside the absolute value in the integrand in \eqref{eq:lem41-int2}
becomes its complex conjugate. Hence its absolute value is unchanged), so
\[
\log M \left( P_{n,k}^{(j)} \right)
=\frac{2}{k} \int_{0}^{1/2} \log \left| (1+e(u))^{k}-\varepsilon_{k}(-1)^{k} e(nu) \right|\,du.
\]

Also, by Jensen's formula in the variable $y$,
\[
\log M (1+x+y)
=\int_{0}^{1} \log^{+} |1+e(u)| \,du
=2\int_{0}^{1/3} \log |1+e(u)| \,du,
\]
where $\log^{+}(z) = \log \max (1, |z|)$ for any $z \in \bbC$. The bounds for
the second integral arise since $1+e(u)=2\cos(\pi u)e^{\pi i u}$ and since for
$0 \leq u \leq 1/2$, we have $2\cos(\pi u) \geq 1$ only when $0 \leq u \leq 1/3$.
 
Therefore
\begin{align}
\label{eq:lem41-CG-1}
\Delta_{n,k}^{(j)}
= & \frac{2}{k} \int_{0}^{1/3} \log \left| 1-\varepsilon_{k}(-1)^{k} \frac{e(nu)}{(1+e(u))^{k}} \right|\,du\\
& +\frac{2}{k} \int_{1/3}^{1/2} \log \left| 1-\varepsilon_{k}(-1)^{k}e(-nu)(1+e(u))^{k} \right|\,du \nonumber.
\end{align}

To obtain these integrals, on $0 \leq u \leq 1/3$, we factored out $(1+e(u))^{k}$,
while on $1/3 \leq u \leq 1/2$, we factored out $e(nu)$.

For $0 \leq u \leq 1/2$,
\[
1+e(u)=2\cos(\pi u)e^{\pi i u}.
\]

We have
\[
(-1)^{k} e^{\pi i(2n-k)/3}=e^{2\pi i(n+k)/3}=e((n+k)/3).
\]

In the first integral in \eqref{eq:lem41-CG-1}, put $u=1/3-t$, and in the second,
put $u=1/3+t$. Using $\varepsilon_{k}(-1)^{k}e(n(1/3-t))/(1+e(1/3-t))^{k}=\varepsilon_{k}e((n+k)/3)r_{1}(t)$
and $\varepsilon_{k}(-1)^{k}e(-n(1/3+t))(1+e(1/3+t))^{k}=\varepsilon_{k}e(-(n+k)/3)r_{2}(t)$,
we obtain \eqref{eq:lem41-CG-3} and the lemma follows.
\end{proof}

Let $S>\pi/3$ be a positive real number. Put
\begin{align*}
R_{1}(S) & = \left\{ x-is/\pi: 0 \leq x \leq 1/3, 0 \leq s \leq S \right\}
\text{ and }\\
R_{2}(S) & = \left\{ x-is/\pi: 0 \leq x \leq 1/6, 0 \leq s \leq S \right\}.
\end{align*}

For any $0<\delta<1/6$, we also put
\[
Q_{\delta}= \left\{ r e^{i\theta}: 0 \leq r<\delta, -\pi/2 \leq \theta \leq 0 \right\}.
\]
The lower bound for $S$ is chosen simply so that $Q_{\delta}$ will always be
inside $R_{1}(S)$ and $R_{2}(S)$.

When $\varepsilon_{k} e(\pm (n+k)/3)=1$ so that we have improper integrals, we
will remove $Q_{\delta}$ from $R_{1}(S)$ and $R_{2}(S)$.

\begin{lemma}
\label{lem:new-2}
{\rm (a)} Let $n \geq 2k$ and write $t=x-is/\pi$ for $s \geq 0$. Then
\begin{equation}
\label{eq:lem31-3a}
\left| r_{1}(t) \right| \leq e^{-(2n-k)s}
\end{equation}
for $0 \leq x \leq 1/3$, and
\begin{equation}
\label{eq:lem31-3b}
\left| r_{2}(t) \right| \leq e^{-((2n-k)-2k/\sqrt{3})s}
\end{equation}
for $0 \leq x \leq 1/6$.

\noindent
{\rm (b)} If $\varepsilon_{k}e((n+k)/3) \neq 1$, then
$\Log \left( 1-\varepsilon_{k}e((n+k)/3)r_{1}(t) \right)$ is holomorphic on
$R_{1}(S)$ for any $S>\pi/3$.

If $\varepsilon_{k}e((n+k)/3)=1$, then
$\Log \left( 1-\varepsilon_{k}e((n+k)/3)r_{1}(t) \right)$ is holomorphic on
$R_{1}(S) \backslash Q_{\delta}$ for any $S>\pi/3$ and any $0<\delta<1/6$.

Similarly, if $\varepsilon_{k}e(-(n+k)/3) \neq 1$, then
$\Log \left( 1-\varepsilon_{k}e(-(n+k)/3)r_{2}(t) \right)$ is holomorphic on
$R_{2}(S)$ for any $S>\pi/3$.

If $\varepsilon_{k}e(-(n+k)/3)=1$, then
$\Log \left( 1-\varepsilon_{k}e(-(n+k)/3)r_{2}(t) \right)$ is holomorphic on
$R_{2}(S) \backslash Q_{\delta}$ for any $S>\pi/3$ and any $0<\delta<1/6$.
\end{lemma}

\begin{proof}
(a) Writing $t=x-is/\pi$ with $s \geq 0$, we have
\[
\left| e^{-i\pi(2n-k)t} \right|=e^{-(2n-k)s}.
\]

If $0 \leq x \leq 1/3$, then
\[
\left| 2 \cos ( \pi/3-\pi t ) \right|^{2}
=4 \left( \cos^{2}(\pi/3-\pi x) + \sinh^{2}(s) \right) \geq 1.
\]

Hence \eqref{eq:lem31-3a} holds.

We proceed similarly for $r_{2}(t)$.
First observe that
\[
\frac{d}{ds} \log|d(s)|
= \frac{2\sinh2s}{2\cosh2s-1} \leq \frac{2}{\sqrt{3}}.
\]

The square of the inequality on the right follows immediately from
$(\cosh(2s)-2)^{2} \geq 0$ and since both sides of the inequality above are
non-negative it follows that the inequality itself is true too.
Integrating from $0$ to $s$ and using $|d(0)|=1$, $d/ds \log|d(s)| \leq 2/\sqrt{3}$
gives $\log |d(s)| \leq 2s/\sqrt{3}$. Hence
\begin{equation}
\label{eq:ds-UB}
|d(s)| \leq e^{2s/\sqrt{3}}.
\end{equation}

If $0 \leq x \leq 1/6$, then $\left| 2 \cos ( \pi/3+\pi t ) \right|
\leq |d(s)| \leq e^{2s/\sqrt{3}}$.
Hence \eqref{eq:lem31-3b} holds.

\vspace*{1.0mm}

(b) We first observe that the two integrals in \eqref{eq:lem41-CG-3} are improper
integrals when $\varepsilon_{k} e(\pm (n+k)/3)=1$ due to a logarithmic singularity
at $t=0$, since $r_{1}(0)=r_{2}(0)=1$. So we must write the two integrals in
\eqref{eq:lem41-CG-3} as
\[
\lim_{\delta \rightarrow 0+} \int_{\delta}^{1/3} \log \left| 1-\varepsilon_{k}e((n+k)/3)r_{1}(t) \right| \,dt
\text{ and }
\]
\[
\lim_{\delta \rightarrow 0+} \int_{\delta}^{1/6} \log \left| 1-\varepsilon_{k}e(-(n+k)/3)r_{2}(t) \right| \,dt.
\]

Since $r_{1}(0)=r_{2}(0)=1$ and $r_{1}'(0)=r_{2}'(0)=
-\pi \left( \sqrt{3} \, k +i(2n-k) \right) \neq 0$, as $t \rightarrow 0$,
\begin{equation}
\label{eq:lem31-2}
1-r_{i}(t)=\pi \left( \sqrt{3} \, k +i(2n-k) \right) t + O \left( t^{2} \right),
\end{equation}
for $i=1,2$. Therefore, $\log  \left| 1-r_{i}(t) \right|=\log (t)+O(1)$, which
is integrable at $0$. Hence both improper integrals in \eqref{eq:lem41-CG-3} converge.

Also
\[
\log \left| 1-\varepsilon_{k} e((n+k)/3) r_{1}(t) \right|
= \Real \left( \Log \left( 1-\varepsilon_{k} e((n+k)/3) r_{1}(t) \right) \right)
\]
and
\[
\log \left| 1-\varepsilon_{k} e(-(n+k)/3) r_{2}(t) \right|
= \Real \left( \Log \left( 1-\varepsilon_{k} e(-(n+k)/3) r_{2}(t) \right) \right).
\]

So we can write
\begin{align}
\label{eq:lem41-CG-3a}
\Delta_{n,k}^{(j)}
= & \frac{2}{k} \Real \left\{ \int_{0}^{1/3} \Log \left( 1-\varepsilon_{k}e((n+k)/3)r_{1}(t) \right) \,dt \right. \\
& \hspace*{7.0mm} \left. + \int_{0}^{1/6} \Log \left( 1-\varepsilon_{k}e(-(n+k)/3)r_{2}(t) \right) \,dt
\right\}. \nonumber
\end{align}

The argument above shows that, when these integrals are improper, their real parts
are integrable. Since $\left| \Imag \Log(z) \right| \leq \pi$, the imaginary
part is bounded and the whole complex integral converges.

Next we want to replace the paths of integration with new paths so we need to
check that the necessary conditions hold.

Since $n \geq 2k$, the upper bounds in both \eqref{eq:lem31-3a} and \eqref{eq:lem31-3b}
are strictly less than $1$ when $s>0$. On $s=0$, $\left| r_{1}(t) \right|$ and
$\left| r_{2}(t) \right|$ are both strictly less than $1$ except at $t=0$, where
$r_{1}(0)=r_{2}(0)=1$.

Since $\left| \varepsilon_{k} e(\pm (n+k)/3) \right|=1$, for $t \in R_{1}(S)$
(recall that $R_{1}(S)$ excludes a small quarter disk around $0$ when
$\varepsilon_{k} e((n+k)/3)=1$) with $t \neq 0$,
\begin{equation}
\label{eq:lem1-1}
\Real \left( 1-\varepsilon_{k} e((n+k)/3) r_{1}(t) \right)
\geq 1-\left| r_{1}(t) \right|>0.
\end{equation}
At $t=0$, in the nonsingular case, $\Real \left( 1- \varepsilon_{k} e(\pm (n+k)/3) \right)
=1-\Real \left( \varepsilon_{k} e(\pm (n+k)/3) \right)>0$. In the singular case,
$t=0$ has been removed.

Similarly, for $t \in R_{2}(S)$ with $t \neq 0$,
\begin{equation}
\label{eq:lem1-2}
\Real \left( 1-\varepsilon_{k} e(-(n+k)/3) r_{2}(t) \right)
\geq 1-\left| r_{2}(t) \right|>0.
\end{equation}

Moreover, $r_{1}(t)$ and $r_{2}(t)$ are holomorphic on $R_{1}(S)$ and $R_{2}(S)$,
respectively: $r_{2}(t)$ is entire, while the estimate in part~(a) shows that
the denominator in $r_{1}(t)$ never vanishes on $R_{1}(S)$.
Therefore, part~(b) follows.
\end{proof}

\begin{lemma}
\label{lem:actual-1}
For relatively prime positive integers $n$ and $k$ with $n \geq 2k$, and
$j \in \{ 1,2,3 \}$,
\begin{equation}
\label{eq:1}
\Delta_{n,k}^{(j)}=\frac{2}{\pi k}\int_{0}^{\infty}
\Imag \left\{ \Log \left( 1-\varepsilon_{k} e^{2 \pi i(n+k)/3}u(s) \right)
+ \Log \left( 1-\varepsilon_{k} e^{-2 \pi i(n+k)/3}v(s) \right) \right\} \, ds.
\end{equation}
\end{lemma}

\begin{proof}
Equations~\eqref{eq:lem1-1} and \eqref{eq:lem1-2} also show that for
$t$ in $R_{1}(S)$ or $R_{2}(S)$, as appropriate,
$1-\varepsilon_{k} e((n+k)/3) r_{1}(t)$ and
$1-\varepsilon_{k} e(-(n+k)/3) r_{2}(t)$ are in the open right half-plane,
so neither of them vanish or cross the branch cut of the principal branch
of the logarithm function.

We are now able to modify both integration paths in \eqref{eq:lem41-CG-3a}
(although starting each path at $\delta>0$ in the case of the singularity at
$t=0$) by applying Cauchy's theorem.

For the first integral in \eqref{eq:lem41-CG-3a}, let $S>\pi/3$. If there is no
singularity at $t=0$, let $\delta'=0$. Otherwise, let $0<\delta<1/6$,
$\delta'=-\delta$ and start with
the path from $\delta$ circling clockwise to $-\delta i$.
\[
\delta' i \rightarrow -iS/\pi \rightarrow 1/3-iS/\pi \rightarrow 1/3,
\]
The value of the first integral in \eqref{eq:lem41-CG-3a}, starting at $\delta>0$
in the case of the singularity at $t=0$, will be the
value of the integral along this path by Cauchy's theorem, since their
paths of integration have the same endpoints.

Similarly, for the second integral, use an initial circular path if there is a
singularity at $t=0$, followed by
\[
\delta' i \rightarrow -iS/\pi \rightarrow 1/6-iS/\pi \rightarrow 1/6.
\]
The value of the second integral in \eqref{eq:lem41-CG-3a}, starting at $\delta>0$
in the case of the singularity at $t=0$, will be the value
of the integral along this path.

Now we consider each of these parts of these paths.

\vspace*{2.0mm}

\noindent
$\bullet$ quarter-circles: recall \eqref{eq:lem31-2},
\[
1-r_{i}(t)=\pi \left( \sqrt{3} \, k +i(2n-k) \right) t + O \left( t^{2} \right),
\]
so $\left| \Log \left( 1-r_{i}(t) \right) \right|=O \left( \left| \log \delta \right| \right)$.
Because the length of this part of the path is $\pi \delta /2$, we find that the
integral over this path is $O \left( \delta \left| \log \delta \right| \right)
\rightarrow 0$ as $\delta \rightarrow 0$.

\vspace*{2.0mm}

\noindent
$\bullet$ left vertical sides: here we can write $t=-is/\pi$. From
\eqref{eq:CG-4b}, we have
\begin{align*}
r_{1}(-is/\pi)
&= \left( 2 \cos \left( \pi/3+is \right) \right)^{-k}e^{-(2n-k)s}\\
&= \left( \cosh (s) -i \sqrt{3} \sinh (s) \right)^{-k}e^{-(2n-k)s}\\
&= \overline{d(s)}^{-k} e^{-(2n-k)s}=u(s).
\end{align*}

Similarly, from \eqref{eq:CG-5b}, we have
\[
r_{2}(-is/\pi)=v(s).
\]

Hence
\[
\Real \int_{\delta'}^{-iS/\pi}
\Log \left( 1- \varepsilon_{k} e((n+k)/3) r_{1}(t) \right) dt
=\frac{1}{\pi} \int_{-\delta' \pi}^{S}
\Imag \left( \Log \left( 1- \varepsilon_{k} e((n+k)/3) u(s) \right) \right) ds
\]
and
\[
\Real \int_{\delta'}^{-iS/\pi}
\Log \left( 1- \varepsilon_{k} e(-(n+k)/3) r_{2}(t) \right) dt
=\frac{1}{\pi} \int_{-\delta' \pi}^{S}
\Imag \left( \Log \left( 1- \varepsilon_{k} e(-(n+k)/3) v(s) \right) \right) ds.
\]

\vspace*{2.0mm}

\noindent
$\bullet$ right vertical sides: the quantities subtracted from $1$ are
$\varepsilon_{k}(-1)^{k}e^{-(2n-k)s} \left( 2 \cosh (s) \right)^{-k}$ and
$\varepsilon_{k}(-1)^{n}e^{-(2n-k)s} \left( 2 \sinh (s) \right)^{k}$.
They are real with absolute value less than $1$. For instance,
\[
\left| \varepsilon_{k}(-1)^{n}e^{-(2n-k)s} \left( 2 \sinh (s) \right)^{k} \right|
\leq \left( e^{-3s}  2 \sinh (s) \right)^{k}
= \left( e^{-2s}-e^{-4s} \right)^{k}<1,
\]
since $2n-k \geq 3k$. So the corresponding principal
logarithms are real. So their vertical integrals have zero real part.

\vspace*{2.0mm}

\noindent
$\bullet$ lower horizontal sides: from \eqref{eq:lem31-3a} and \eqref{eq:lem31-3b}, we have
\[
\left| r_{1}(t) \right| \leq e^{-(2n-k)S}
\quad \text{ and } \quad
\left| r_{2}(t) \right| \leq e^{-((2n-k)-2k/\sqrt{3})S}.
\]

For large $S$, these are both at most $1/2$ and go to $0$ exponentially fast
with $S$. For any $z$ with $|z|<1/2$, we have
\[
\left| \Log (1-z) \right| \leq \frac{|z|}{1-|z|} \leq 2|z|.
\]
Applied here, we see that the absolute value of the integrands goes to $0$
exponentially fast with $S$. Since the horizontal segments have fixed lengths,
both lower horizontal integrals go to $0$ as $S \rightarrow \infty$.

We now combine the above information.
In the improper integral case, we first let $\delta \rightarrow 0^{+}$.
In both cases, we then let $S \rightarrow \infty$.
So the only contribution comes from the left vertical sides, from which we
obtain \eqref{eq:1}.
\end{proof}

\section{An Intermediate Estimate}
\label{sect:intermediate}

Put
\[
w(s)=e^{-((2n-k)-i\sqrt{3}k)s}.
\]

Let
\[
\Lambda_{n,k}^{(j)}=\frac{2}{\pi k}\int_{0}^{\infty}
\Imag \left\{ \Log \left( 1-\varepsilon_{k} e^{2 \pi i(n+k)/3}w(s) \right)
+ \Log \left( 1-\varepsilon_{k} e^{-2 \pi i(n+k)/3}w(s) \right) \right\} \, ds.
\]
That is, $\Lambda_{n,k}^{(j)}$ denotes the expression \eqref{eq:1} with both $u$
and $v$ replaced by $w$.

The rationale for the use of $w(s)$ comes from the series expansions of $u(s)$
and $v(s)$ around $s=0$. The expansions for $u(s)$ and $v(s)$
start similarly, $u(s)-v(s)=-4ks^{2}+O \left( s^{3} \right)$. $w(s)$ is between
them with $u(s)-w(s)=-2ks^{2}+O \left( s^{3} \right)$ and
$v(s)-w(s)=2ks^{2}+O \left( s^{3} \right)$. Moreover, since the $s^{2}$ coefficients
in these differences are negatives of one another, they give
$2qks^{2}w(s)^{q} \left( e^{-iq\theta}-e^{iq\theta} \right)$, where
$\theta=2\pi (n+k)/3$, which directly
gives the $n^{-3}$ terms in Theorem~\ref{thm:2.1}.

\begin{lemma}
\label{lem:est1}
For relatively prime positive integers $n$ and $k$ with $n \geq 2k$, and
$j \in \{ 1,2,3 \}$,
\begin{equation}
\label{eq:3a}
\Delta_{n,k}^{(j)}-\Lambda_{n,k}^{(j)}
= \frac{16}{\pi} \Real \left( \frac{1}{(2n-k-i\sqrt{3}k)^{3}} \right)
\sum_{q \geq 1} \frac{\varepsilon_{k}^{q} \sin (2\pi q(n+k)/3)}{q^{3}}
+R_{n,k}^{(j)},
\end{equation}
where
\begin{equation}
\label{eq:3b}
\left| R_{n,k}^{(j)} \right|
\leq \frac{156}{(2n-k)^{4}}.
\end{equation}
\end{lemma}

\begin{proof}
Because $\Real(d(s))=\cosh(s)>0$, the principal logarithm
\[
h(s)=\Log(d(s))
\]
is well-defined for $s \geq 0$. Direct differentiation gives
\[
h'(s)=\frac{d'(s)}{d(s)},
\qquad
h''(s)=\frac{4}{d(s)^{2}},
\qquad
h'''(s)=-\frac{8d'(s)}{d(s)^{3}}.
\]

Moreover,
\[
|d(s)|^{2}
=\cosh^{2}s+3\sinh^{2}s
=1+4\sinh^{2}s,
\]
and
\[
|d'(s)|^{2}
=\sinh^{2}s+3\cosh^{2}s
=3+4\sinh^{2}s.
\]
Hence
\[
\left| h''(s) \right| \leq 4,
\]
and
\[
\left| h'''(s) \right|
= 8\frac{\sqrt{3+4\sinh^{2}s}}{(1+4\sinh^{2}s)^{3/2}}
\leq 8\sqrt{3}.
\]

Since
\[
h(0)=0, \qquad h'(0)=i\sqrt{3}, \qquad h''(0)=4,
\]
Taylor's theorem gives
\begin{equation}
\label{eq:lem51-1}
\left| h(s)-i\sqrt{3}\,s-2s^{2} \right|
\leq \frac{4\sqrt{3}}{3}s^{3}.
\end{equation}

Using only $|h''(s)| \leq 4$ also gives
\begin{equation}
\label{eq:lem51-2}
\left| h(s)-i\sqrt{3}\,s \right| \leq 2s^{2}.
\end{equation}

Set
\[
g_{+}(s)=h(s)-i\sqrt{3}\,s,
\qquad
g_{-}(s)=-\overline{h(s)}-i\sqrt{3}\,s.
\]
Then
\[
g_{-}(s)=-\overline{g_{+}(s)},
\]
and
\begin{equation}
\label{eq:lem51-4}
v(s)=w(s)e^{kg_{+}(s)}, \qquad u(s)=w(s)e^{kg_{-}(s)}.
\end{equation}

By \eqref{eq:lem51-1},
\begin{equation}
\label{eq:lem51-5}
\left| g_{+}(s)-2s^{2} \right|, \ \left| g_{-}(s)+2s^{2} \right| \leq \frac{4\sqrt{3}}{3}s^{3},
\end{equation}
while \eqref{eq:lem51-2} gives
\begin{equation}
\label{eq:lem51-6}
\left| g_{\pm}(s) \right| \leq 2s^{2}.
\end{equation}

Furthermore,
\begin{equation}
\label{eq:lem51-7}
\Real \left( g_{+}(s) \right)=\log|d(s)|
\quad \text{ and } \quad \Real \left( g_{-}(s) \right) =-\log|d(s)| \leq 0.
\end{equation}

For $s>0$,
\[
|u(s)| \leq e^{-(2n-k)s}, \qquad |w(s)|=e^{-(2n-k)s},
\]
and, by \eqref{eq:ds-UB},
\[
|v(s)| \leq e^{-((2n-k)-2k/\sqrt{3})s}.
\]

Since $n \geq 2k$, we have $2n-k \geq 3k$, and therefore
\[
(2n-k)-\frac{2k}{\sqrt{3}}>0.
\]

Consequently the logarithmic series may be integrated termwise, because
\begin{align*}
&\sum_{q \geq 1} \frac{1}{q}
\int_{0}^{\infty}
\left( |u(s)|^{q}+|v(s)|^{q}+2|w(s)|^{q} \right) \, ds\\
&\qquad \leq \zeta(2) \left( \frac{3}{2n-k} + \frac{1}{(2n-k)-2k/\sqrt{3}} \right)
<\infty. \nonumber
\end{align*}

Write
\[
\vartheta=\frac{2\pi(n+k)}{3}.
\]

Using the logarithmic series in the exact contour identity and in the definition
of $\Lambda_{n,k}^{(j)}$, and using \eqref{eq:lem51-4}, we obtain
\begin{align}
\label{eq:lem51-9}
\Delta_{n,k}^{(j)}-\Lambda_{n,k}^{(j)}
= &-\frac{2}{\pi k} \sum_{q \geq 1}\frac{\varepsilon_{k}^{q}}{q}
\int_{0}^{\infty} \Imag \left[ w(s)^{q} \right. \\
& \left. \qquad \times
\left\{
 e^{iq\vartheta}(e^{qkg_{-}(s)}-1)
 +e^{-iq\vartheta}(e^{qkg_{+}(s)}-1)
\right\}
\right] \, ds. \nonumber
\end{align}

For every complex number $z$,
\[
e^{z}-1-z=z^{2} \int_{0}^{1} (1-t)e^{tz} \, dt,
\]
and hence
\begin{equation}
\label{eq:lem51-10}
\left| e^{z}-1-z \right| \leq \frac{|z|^{2}}{2} e^{\max(\Real(z),0)}.
\end{equation}

For $g_{-}(s)$, \eqref{eq:lem51-5}--\eqref{eq:lem51-7} and \eqref{eq:lem51-10} give
\begin{align}
\label{eq:lem51-11}
\left| e^{qkg_{-}(s)}-1+2qks^{2} \right|
& \leq qk \left| g_{-}(s)+2s^{2} \right|
       + \left| e^{qkg_{-}(s)}-1-qkg_{-}(s) \right| \\
& \leq \frac{4\sqrt{3}}{3}qks^{3}
       + \frac{1}{2}q^{2}k^{2}|g_{-}(s)|^{2} \nonumber \\
& \leq \frac{4\sqrt{3}}{3}qks^{3} +2q^{2}k^{2}s^{4}. \nonumber
\end{align}

For $g_{+}(s)$, \eqref{eq:ds-UB}, \eqref{eq:lem51-5}, \eqref{eq:lem51-6}
and \eqref{eq:lem51-10} give
\begin{align}
\label{eq:lem51-12}
\left| e^{qkg_{+}(s)}-1-2qks^{2} \right|
& \leq qk \left| g_{+}(s)-2s^{2} \right|
       + \left| e^{qkg_{+}(s)}-1-qkg_{+}(s) \right| \\
& \leq \frac{4\sqrt{3}}{3}qks^{3} + \frac{1}{2} q^{2}k^{2} \left| g_{+}(s) \right|^{2} e^{qk \Real(g_{+}(s))} \nonumber \\
& \leq \frac{4\sqrt{3}}{3}qks^{3} + 2q^{2}k^{2}s^{4}e^{2qks/\sqrt{3}}. \nonumber
\end{align}

The quadratic terms extracted in \eqref{eq:lem51-11} and \eqref{eq:lem51-12} contribute
\begin{align}
\label{eq:lem51-13}
&-\frac{2}{\pi k}
\sum_{q \geq 1}\frac{\varepsilon_{k}^{q}}{q}
\int_{0}^{\infty}
\Imag \left[
2qks^{2}w(s)^{q}
\left( e^{-iq\vartheta}-e^{iq\vartheta} \right)
\right] ds\\
= & \frac{8}{\pi}
\sum_{q \geq 1}
\varepsilon_{k}^{q}\sin(q\vartheta)
\int_{0}^{\infty}
s^{2}\Real \left( w(s)^{q} \right) \, ds. \nonumber
\end{align}

Since
\[
w(s)^{q}=e^{-q(2n-k-i\sqrt{3}\,k)s},
\]
we have
\begin{equation}
\label{eq:lem51-14}
\int_{0}^{\infty} s^{2}w(s)^{q}\,ds
=\frac{2}{q^{3}(2n-k-i\sqrt{3}\,k)^{3}}.
\end{equation}

Substituting \eqref{eq:lem51-14} into \eqref{eq:lem51-13} gives
\[
\frac{16}{\pi}
\Real \left( \frac{1}{(2n-k-i\sqrt{3}\,k)^{3}} \right)
\sum_{q \geq 1}
\frac{\varepsilon_{k}^{q}\sin(q\vartheta)}{q^{3}},
\]
which is the asserted main term.

It remains to bound the contribution of the errors in \eqref{eq:lem51-11}
and \eqref{eq:lem51-12}. From \eqref{eq:lem51-9}, \eqref{eq:lem51-11} and
\eqref{eq:lem51-12},
\begin{align}
\label{eq:lem51-15}
\left| R_{n,k}^{(j)} \right|
\leq &
\frac{2}{\pi k}
\sum_{q \geq 1}\frac{1}q
\int_{0}^{\infty} e^{-q(2n-k)s}\\
&\quad \times
\left[
\frac{8\sqrt{3}}{3}qks^{3}
+2q^{2}k^{2}s^{4}
+2q^{2}k^{2}s^{4}e^{2qks/\sqrt{3}}
\right]ds. \nonumber
\end{align}

For $\lambda>0$,
\begin{equation}
\label{eq:lem51-16}
\int_{0}^{\infty} s^{3}e^{-q\lambda s}\,ds
=\frac{6}{q^{4}\lambda^{4}}
\quad \text{ and } \quad
\int_{0}^{\infty} s^{4}e^{-q\lambda s}\,ds
= \frac{24}{q^{5}\lambda^{5}}.
\end{equation}

Applying \eqref{eq:lem51-16} to \eqref{eq:lem51-15} gives
\begin{align}
\label{eq:lem51-17}
\left| R_{n,k}^{(j)} \right|
\leq & \frac{32\sqrt{3}\,\zeta(4)}{\pi(2n-k)^{4}}\\
     & + \frac{96k\zeta(4)}{\pi}
         \left\{ \frac{1}{(2n-k)^{5}} + \frac{1}{((2n-k)-2k/\sqrt{3})^{5}} \right\}. \nonumber
\end{align}

Finally, since $2n-k \geq 3k$,
\[
k \leq \frac{2n-k}{3},
\]
and
\[
(2n-k)-\frac{2k}{\sqrt{3}}
\geq \left(1-\frac{2}{3\sqrt{3}}\right)(2n-k).
\]

Thus \eqref{eq:lem51-17} yields
\begin{align*}
\left| R_{n,k}^{(j)} \right|
&\leq
\frac{32\zeta(4)}{\pi(2n-k)^{4}}
\left[ \sqrt{3}+1+\left(1-\frac{2}{3\sqrt{3}}\right)^{-5} \right]\\
&= \frac{155.3266\ldots}{(2n-k)^{4}}
< \frac{156}{(2n-k)^{4}}.
\end{align*}

Finally, for real $a,b$,
\[
\Real \left( \frac{1}{(a-ib)^{3}} \right)
=\frac{a(a^{2}-3b^{2})}{(a^{2}+b^{2})^{3}}.
\]

Taking $a=2n-k$ and $b=\sqrt{3}\,k$ gives
\[
\Real \left( \frac{1}{(2n-k-i\sqrt{3}\,k)^{3}} \right)
= \frac{(2n-k)\bigl((2n-k)^{2}-9k^{2}\bigr)}
{\bigl((2n-k)^{2}+3k^{2}\bigr)^{3}}.
\]
This completes the proof.
\end{proof}

\section{Proof of Theorem~\ref{thm:2.1}}
\label{sect:5}

\begin{lemma}
\label{lem:5.1}
Let $n$ and $k$ be relatively prime positive integers with $n \geq 2k$ and let
$j \in \{ 1,2,3 \}$. Then
\[
\Lambda_{n,k}^{(j)}
= \left\{
\begin{array}{ll}
-\dfrac{2\pi \sqrt{3}}{3((2n-k)^{2}+3k^{2})} & \text{if $\varepsilon_{k}=1$ and $3|(n+k)$},\\[2ex]
 \dfrac{2\pi \sqrt{3}}{9((2n-k)^{2}+3k^{2})} & \text{if $\varepsilon_{k}=1$ and $3 \nmid (n+k)$},\\[2ex]
 \dfrac{\pi \sqrt{3}}{3((2n-k)^{2}+3k^{2})}  & \text{if $\varepsilon_{k}=-1$ and $3|(n+k)$},\\[2ex]
-\dfrac{\pi \sqrt{3}}{9((2n-k)^{2}+3k^{2})}  & \text{if $\varepsilon_{k}=-1$ and $3 \nmid (n+k)$}.
\end{array}
\right.
\]
\end{lemma}

\begin{proof}
Consider the series
\[
\Log (1-z) = - \sum_{q \geq 1} \frac{z^{q}}{q}
\]
for $|z|<1$. We want to apply this to evaluate $\Lambda_{n,k}^{(j)}$.

Termwise integration of the logarithm series is justified by
\[
\sum_{q \geq 1} \frac{1}{q} \int_{0}^{\infty} e^{-(2n-k)qs} \, ds
= \frac{\zeta(2)}{2n-k}.
\]

We have
\[
\int_{0}^{\infty} w(s)^{q} ds = \frac{1}{q(2n-k-i\sqrt{3} \, k)},
\]
and hence
\[
\Imag \left( \int_{0}^{\infty} w(s)^{q} ds \right)
= \frac{\sqrt{3} \, k}{q((2n-k)^{2}+3k^{2}}.
\]

Substituting these into
\[
\Lambda_{n,k}^{(j)}
= -\frac{4}{\pi k} \sum_{q \geq 1} \frac{\varepsilon_{k}^{q} \cos(2q\pi (n+k)/3)}{q}
\Imag \left( \int_{0}^{\infty} w(s)^{q} ds \right)
\]
gives
\begin{equation}
\label{eq:4}
\Lambda_{n,k}^{(j)}=-\frac{4\sqrt{3}}{\pi((2n-k)^{2}+3k^{2})}
\sum_{q \geq 1} \frac{\varepsilon_{k}^{q}\cos(2q\pi (n+k)/3)}{q^{2}}.
\end{equation}

If $\varepsilon_{k}=1$ and $3|(n+k)$, then the sum is $\zeta(2)=\pi^{2}/6$.

If $\varepsilon_{k}=1$ and $3 \nmid (n+k)$, then $n+k \equiv \pm 1 \pmod{3}$,
so $\cos(2q\pi (n+k)/3)=\cos(2q\pi/3)$. Here
\[
\sum_{q \geq 1} \frac{\cos(2\pi q/3)}{q^{2}}
=\sum_{q \geq 1, 3|q}q^{-2}-(1/2)\sum_{q \geq 1, 3 \nmid q}q^{-2}
=\frac{1}{9}\zeta(2)-\frac{1}{2}\frac{8}{9}\zeta(2)=-\frac{\pi^{2}}{18}.
\]

Similarly, for $\varepsilon_{k}=-1$, the sum in \eqref{eq:4} is
$-\pi^{2}/12$ or $\pi^{2}/36$ for $3|(n+k)$ or $3 \nmid (n+k)$, respectively.
\end{proof}

\begin{proof}[Proof of Theorem~\ref{thm:2.1}]
We start with proving the series expansions in Theorem~\ref{thm:2.1}.

Recalling that $n+k \equiv \epsilon \pmod{3}$ with $\epsilon \in \{ -1, 0, 1 \}$,
we evaluate the sum in \eqref{eq:3a}:
\[
\sum_{q \geq 1} \frac{\varepsilon_{k}^{q} \sin (2\pi q(n+k)/3)}{q^{3}}
= \left\{
\begin{array}{cl}
0 & \text{if $\epsilon=0$},\\[1ex]
 \dfrac{2\pi^{3}\epsilon}{81}  & \text{if $\varepsilon_{k}=1$ and $\epsilon=\pm 1$},\\[2ex]
-\dfrac{5\pi^{3}\epsilon}{162} & \text{if $\varepsilon_{k}=-1$ and $\epsilon=\pm 1$}.
\end{array}
\right.
\]

We use formula~1.443.5 in \cite{GR} for this where $k$ there is our $q$ here.
We take $x=2\pi/3$ for
$\left( \varepsilon_{k}, \epsilon \right)=(1,1)$;
$x=4\pi/3$ for $\left( \varepsilon_{k}, \epsilon \right)=(1,-1)$;
$x=5\pi/3$ for $\left( \varepsilon_{k}, \epsilon \right)=(-1,1)$; and
$x=\pi/3$ for $\left( \varepsilon_{k}, \epsilon \right)=(-1,-1)$. For the last
two (with $\varepsilon_{k}=-1$), we use the identity $(-1)^{q} \sin (q\theta)
=\sin(q(\theta+\pi))$.

Combining this with Lemmas~\ref{lem:est1} and \ref{lem:5.1}, we get the following
exact expressions for $\Delta_{n,k}^{(j)}$:
\[
\Delta_{n,k}^{(1)}
= \left\{
\begin{array}{ll}
-\dfrac{\pi \sqrt{3}}{6(n^{2}-nk+k^{2})}+R_{n,k}^{(1)} & \text{if $3|(n+k)$},\\
 \dfrac{\pi \sqrt{3}}{18(n^{2}-nk+k^{2})}
+\dfrac{32\pi^{2}\epsilon}{81} \Real \left( \dfrac{1}{(2n-k-i\sqrt{3}\, k)^{3}} \right)
+R_{n,k}^{(1)} & \text{if $3 \nmid (n+k)$}
\end{array}
\right.
\]
and for $j=2,3$,
\[
\Delta_{n,k}^{(j)}
= \left\{
\begin{array}{ll}
 \dfrac{\pi \sqrt{3}}{12(n^{2}-nk+k^{2})}+R_{n,k}^{(j)} & \text{if $3|(n+k)$},\\
-\dfrac{\pi \sqrt{3}}{36(n^{2}-nk+k^{2})}
-\dfrac{40\pi^{2}\epsilon}{81} \Real \left( \dfrac{1}{(2n-k-i\sqrt{3}\, k)^{3}} \right)
+R_{n,k}^{(j)} & \text{if $3 \nmid (n+k)$}.
\end{array}
\right.
\]

Consider the non-numerical factor in the first term in these expansions. Since
\[
\frac{1}{n^{2}-nk+k^{2}}=\frac{1}{n^{2}}+\frac{k}{n^{3}}
-\frac{k^{3}}{n^{3}(n^{2}-nk+k^{2})}
\]
and $n^{2}-nk+k^{2} \geq 3n^{2}/4$ for $n \geq 2k$, it follows that
\begin{equation}
\label{eq:thm21-1}
\left| \frac{1}{n^{2}-nk+k^{2}}-\frac{1}{n^{2}}-\frac{k}{n^{3}} \right|
\leq \frac{4k^{3}}{3n^{5}} \leq \frac{2k^{2}}{3n^{4}},
\end{equation}
as required for our expansion with a $O \left( k^{2}/n^{4} \right)$ error term.

Similarly, we can write
\[
\Real \left( \frac{1}{(2n-k-i\sqrt{3} \, k)^{3}} \right)
= \frac{4(2n-k)(n-2k)(n+k)}{64(n^{2}-nk+k^{2})^{3}}.
\]

With $x=k/n$, we find that
\begin{align*}
& n^{3}
\left( \Real \left( \frac{1}{(2n-k-i\sqrt{3} \, k)^{3}} \right)
- \left( \frac{1}{8n^{3}} + \frac{3k}{16n^{4}} \right) \right) \\
= & -x^{2}\frac{3x^{5}-7x^{4}+12x^{3}-9x^{2}+2x+6}{16(x^{2}-x+1)^{3}}.
\end{align*}

The polynomial in the numerator is positive for $0 \leq x \leq 1/2$, since
$3x^{5}-7x^{4}+12x^{3}-9x^{2}+2x+6 \geq 6-7(1/2)^{4}-9(1/2)^{2}=53/16>0$.
The polynomial in the denominator has no real roots. Hence
\[
\Real \left( \frac{1}{(2n-k-i\sqrt{3} \, k)^{3}} \right)
<\frac{1}{8n^{3}} + \frac{3k}{16n^{4}}.
\]

Similarly, with $x=k/n$, we find that
\begin{align*}
& n^{3}
\left( \Real \left( \frac{1}{(2n-k-i\sqrt{3} \, k)^{3}} \right)
- \left( \frac{1}{8n^{3}} - \frac{k}{4n^{4}} \right) \right) \\
= & x(x-1)(2x-1)\frac{2x^{4}-4x^{3}+8x^{2}-6x+7}{16(x^{2}-x+1)^{3}},
\end{align*}
which is non-negative for $0<x \leq 1/2$. Combining these two inequalities, we obtain
\begin{equation}
\label{eq:thm21-2}
\left| \Real \left( \frac{1}{(2n-k-i\sqrt{3} \, k)^{3}} \right)
- \frac{1}{8n^{3}} \right|
\leq \frac{k}{4n^{4}}.
\end{equation}

Finally, Lemma~\ref{lem:est1}, along with $2n-k \geq 3n/2$ gives
\begin{equation}
\label{eq:thm21-3}
\left| R_{n,k}^{(j)} \right|
\leq \frac{156}{(2n-k)^{4}}
\leq 156 \left( \frac{2}{3} \right)^{4} \frac{1}{n^{4}}
=\frac{2496}{81n^{4}}<31 \frac{k^{2}}{n^{4}}.
\end{equation}

Substituting \eqref{eq:thm21-1}, \eqref{eq:thm21-2} and
\eqref{eq:thm21-3} into our expressions above for $\Delta_{n,k}^{(j)}$ gives
us the series expansions in Theorem~\ref{thm:2.1}.

\vspace*{1.0mm}

We turn now to the proof of Conjecture~\ref{conj:1.3}.

From our evaluation earlier in the proof of the sum in \eqref{eq:3a}, we have
\[
\left| \sum_{q \geq 1} \frac{\varepsilon_{k}^{q} \sin (2\pi q(n+k)/3)}{q^{3}} \right|
\leq \frac{5\pi^{3}}{162}.
\]

Also
\[
\left| \Real \left( \frac{1}{(2n-k-i\sqrt{3}k)^{3}} \right) \right|
\leq \frac{1}{(2n-k)^{3}}.
\]

Therefore, Lemma~\ref{lem:est1} gives
\[
\left| \Delta_{n,k}^{(j)}-\Lambda_{n,k}^{(j)} \right|
\leq \frac{40\pi^{2}}{81(2n-k)^{3}}+\frac{156}{(2n-k)^{4}}.
\]

If $2n-k \geq 36$, then
\[
\frac{40\pi^{2}}{81(2n-k)^{3}}+\frac{156}{(2n-k)^{4}}
\leq \left( \frac{40\pi^{2}}{81 \cdot 36}+\frac{156}{36^{2}} \right) \frac{1}{(2n-k)^{2}}
<0.256 \frac{1}{(2n-k)^{2}}.
\]

Because $2n-k \geq 3k$,
\[
\left| \Lambda_{n,k}^{(j)} \right|
\geq \frac{\pi\sqrt{3}}{9((2n-k)^{2}+3k^{2})}
\geq \frac{\pi\sqrt{3}}{12(2n-k)^{2}}>\frac{0.453}{(2n-k)^{2}}.
\]

So $\Delta_{n,k}^{(j)}$ has the same sign as $\Lambda_{n,k}^{(j)}$.
From Lemma~\ref{lem:5.1}, $\Lambda_{n,k}^{(j)}$ has exactly the sign required
in each part of Conjecture~\ref{conj:1.3}. This completes the proof of
Conjecture~\ref{conj:1.3} when
$2n-k \geq 36$, in particular whenever $n \geq 24$.

It remains to prove Conjecture~\ref{conj:1.3} for $2 \leq n<24$. We do so by
calculation.

There are 172 choices of $\left( n,k,\varepsilon_{k} \right)$ with $2 \leq n<24$,
$1 \leq k \leq n/2$, and $\gcd(n,k)=1$. For each of these polynomials, the Mahler
measure was calculated using PARI/GP. No counterexamples to Conjecture~\ref{conj:1.3}
were found. This completes the proof.
\end{proof}

\bibliographystyle{amsplain}

\end{document}